\documentclass[eqno,a4paper,11pt]{article}
\title{Remark on global existence of solutions to the 1D compressible Euler equation with  time-dependent damping}
 \author{Yuusuke Sugiyama\footnote{e-mail:sugiyama.y@e.usp.ac.jp}\ \ \\ 
School of Engineering, The University of Shiga Prefecture\\
 2500, Hassaka-cho, Hikone-City, Shiga 522-8533 Japan}
\date{}
\usepackage{amssymb,amsmath,amsthm,mathrsfs,delarray,enumerate}
\usepackage{graphics}

\theoremstyle{definition} 
\newtheorem{Def}{Deffinition}[section]
\newtheorem{Prop}[Def]{Proposition}
\newtheorem{Lemma}[Def]{Lemma}
\newtheorem{Thm}[Def]{Theorem}
\newtheorem{Remark}[Def]{Remark}

\newcommand{\R}{\mathbb{R}} 

\makeatletter

\@addtoreset{equation}{section}

\begin{document}
\maketitle 
\begin{abstract}
In this paper, we consider the 1D  compressible Euler equation with  the damping coefficient $\lambda/(1+t)^{\mu}$.
Under the assumption that $0\leq \mu <1$ and $\lambda >0$ or $\mu=1$ and $\lambda > 2$,
we prove that solutions exist globally in time, if  initial data are small $C^1$ perturbation near constant states.
In particular, we remove the conditions on the limit $\lim_{|x| \rightarrow  \infty} (u (0,x), v (0,x))$, assumed in previous results.
\end{abstract}

\section{Introduction}
In this paper, we consider the following Cauchy problem of the compressible Euler equation with time-dependent  damping
\begin{eqnarray} 
\left\{  \begin{array}{ll} \label{de0}
  u_t - v_x  =0,\\
v_t + p(u)_x = -\dfrac{\lambda}{(1+t)^\mu} v , \\   
   (u(0,x), v(0,x))=(1+\varepsilon \phi (x), \varepsilon \psi (x)).
\end{array} \right.  
\end{eqnarray} 
Here $x \in \R$ is the Lagrangian spatial variable and $t \in \R_+$ is time. $u=u(t,x)$ and $v=v(t,x)$  are the real valued unknown functions, which stand for the specific volume and the fluid  velocity. $\varepsilon $ is a small positive constant.
Throughout this paper, we assume that $\lambda \geq 0$ and $\mu \geq0$.
In the case with $ \lambda \equiv 0$, the equations in \eqref{de0} are the compressible Euler, which is a fundamental model for the compressible inviscid fluids.
In the case with $\mu=0$ and $\lambda >0$ (constant damping), this system describes the flow of fluids in  porous media.
We assume that the flow is barotropic ideal gases. Namely the pressure $p$ satisfies that 
\begin{eqnarray}\label{poly}
p(u)=\frac{u^{-\gamma}}{\gamma} \ \  \mbox{for} \  \gamma > 1. 
\end{eqnarray}
For initial data, in order to avoid the singularity of $p'$, we assume that there exists  constant $\delta_0 >0$  such that for all $x \in \R$
\begin{eqnarray} \label{nosin}
u (0, x) \geq \delta_0.
\end{eqnarray}
The local existence and the uniqueness theorem for \eqref{de0}  with $C^1 _b$ initial data  is proved by   Douglis \cite{D} and Hartman and Winter \cite{HW} (see also Majda \cite{m} and Courant and Lax \cite{CL}) as follows.
\begin{Prop}\label{local}
If $u(0,\cdot), v(0,\cdot) \in C^1 _b (\R)$ and \eqref{nosin} is assumed for some $\delta_0 >0$, then \eqref{de0} has a local and unique solution $(u,v)$ satisfying
$$
(u,v) \in C^1 _b ([0,T]\times \R) \times C^1 _b ([0,T]\times \R) 
$$
and
$$u(t,x) \geq \delta_0 /2 \ \mbox{on} \ [0,T]\times \R$$
for some $T=T(\|u(0,\cdot) \|_{C^1 _b (\R)}, \|v(0,\cdot) \|_{C^1 _b (\R)}, \delta_0)$.
\end{Prop}
In \cite{D, HW}, the local existence and uniqueness theorem is proved in a trapezoid domain. Joining the solutions, one can construct  the unique solution on $[0,T]\times \R$ as pointed out in Majda \cite{m}.
We note that the assumption \eqref{nosin} is satisfied if $\varepsilon $ is sufficiently small for fixed $\phi \in C^1 _b (\R)$.
Before recalling known results, we give notations. We set $c=\sqrt{-p'(u)}$ and $\eta= \int_{u} ^{\infty} c(\xi) d\xi=\frac{2}{\gamma-1} u^{-(\gamma-1)/2}$
and define Riemann invariants as follows:
\begin{eqnarray} \label{ri}
\begin{array}{ll}
r = v-\eta + \frac{2}{\gamma-1}, \\
s = v+\eta - \frac{2}{\gamma-1}.
\end{array}
\end{eqnarray}
For the solution constructed by Proposition \ref{local}, we define its lifespan  by
\begin{align} \label{t*}
T^* =& \sup\{T>0 \ | \   \sup_{t \in [0,T )} \| (u,v)(t)\|_{L^{\infty}} +  \| (u_t,v_t)(t)\|_{L^{\infty}} \\
&+\| (u_x,v_x)(t)\|_{L^{\infty}} +  \| p'(u(t))\|_{L^{\infty}} < \infty \}. \notag
\end{align}
We note that the boundedness of $p'(u)$ means the positivity of $u$.
From the definition of $T^*$, one may consider that there are several types of the breakdown of solutions such as  the divergence of $u$, $v$, $p'$ and time or space derivative of the solution.
However the divergence of  $u$, $v$ and $p'$ does not occur in the small data regime, since we can show that solutions are small if initial data are small (see Lemma \ref{esP}).   
Hence, for the classical solution of \eqref{de0} with small data, "blow-up'' means always the divergence of the time or space derivative of the solution.

In the case with $\lambda \equiv 0$ (no damping case),  for more general $2\times 2$ strictly hyperbolic system including the 1D Euler equation, sufficient conditions for the blow-up (the formation singularity) has been studied by many mathematicians
(e.g. Lax \cite{lax}, Zabusky \cite{z} and Klainerman and Majda \cite{km}).  Applying these blow-up results, we can show that  if $r_x$ or $s_x$ is negative at some point, then the derivative blow-up occurs under suitable assumptions on $r(0,x)$ and $s(0,x)$.
 Namely there are blow-up solutions even if initial data are small perturbations near constant states. While if $r_x (0,x), s_x (0,x) \geq 0$ and $\varepsilon $ is suitably small, 
 the global solution exists (see Remark \ref{decr}).
 
 In the case with $\mu=0$ and $\lambda >0$ (constant damping case), Hsiao and  Liu \cite{HL} has proved that  classical solutions exist globally in time,  if initial data are small perturbations near constant states and that  small solutions asymptotically behave like
that to the following porous media system as $t\rightarrow \infty$:
\begin{eqnarray*} 
\left\{  \begin{array}{ll} 
 \bar{u}_t = -p(\bar{u})_{xx},\\
 \bar{v}=-p(\bar{u})_x .
\end{array} \right.  
\end{eqnarray*} 
After Hsiao and  Liu's work, many improvements and generalizations of this work have been investigated (see Hsiao and  Liu \cite{HL2}, Nishihara \cite{NK}, Hsiao and Serre \cite{HS}, 
Marcati  and Nishihara \cite{MN} and Mei \cite{Mei}).  
We note that, in the above papers  for the 1D Euler equation with constant damping, they assume the existence of the limit $\lim_{x \rightarrow \pm \infty} (u_0 (x), v_0 (x))$ and that the convergence rate of the limit
is sufficiently fast in order to show the global existence of solutions via $L^2$ energy estimates.

In  \cite{XP1, XP2, XP3}, Pan has found thresholds of $\mu$ and $\lambda$ separating the existence and the nonexistence of global solution of \eqref{de0} in small data regime.
Namely, in the case with $0\leq \mu <1$ and $\lambda >0$ or $\mu=1$ and $\lambda >2$,  Pan \cite{XP2} has proved  that solutions  of \eqref{de0} exist globally in time,  if initial data are small and compact perturbations of  constant states. While, in the case with $\mu >1$ and $\lambda >0$ or $\mu=1$ and $0 \leq \lambda \leq 2$, Pan  has proved that solutions of \eqref{de0} can blow up under some conditions on initial data in \cite{XP1, XP3}. 
In \cite{CYZZ}, Cui, Yin, Zhang and Zhu have proved that the global solution of \eqref{de0} with $0\leq \mu<1$ and $\lambda >0$ asymptotically behaves like that to the corresponding  porous media system having a time-dependent coefficient, if $\varepsilon>0$ is small.  In \cite{XP1, XP2, XP3, CYZZ}, they also assume that the convergence rate of the limit  $\lim_{x \rightarrow \pm \infty} (u_0 (x), v_0 (x))$ is sufficiently fast, since their method is based on the $L^2$ energy method. 
In \cite{s5}, the author has  shown the global existence with small $\varepsilon >0$, assuming the existence of the one side limit $\lim_{x \rightarrow - \infty} (u_0 (x), v_0 (x))$.
In this paper, we show the global existence without any assumption on the behavior of initial data at spacial infinity, only assuming that initial data are small perturbation near constant states.
The difference between the proofs of the main theorem of this paper and \cite{s5} is  the definition of the Riemann invariant.
The Riemann invariant used in \cite{s5} can not be small, when $(u,v)$ is close to $(0,1)$. 
The proof of the main theorem of this paper is somewhat simpler than that in \cite{s5}.

Now we state the main theorem of this paper.
\begin{Thm}\label{maing}
Let   $\gamma >1$, $( \phi,  \psi) \in C^1 _b (\R)$. Suppose that $0\leq \mu <1$ and $\lambda >0$ or $\mu=1$ and $\lambda >2$. There exists a number $\varepsilon _0 >0$ such that if $0< \varepsilon \leq \varepsilon_0 $, then the Cauchy problem \eqref{de0} has a global unique solution satisfying the following decay estimate:
\begin{eqnarray} \label{dec}
\|(r_t (t),s_t (t))\|_{L^\infty} + \|(r_x (t),s_x (t))\|_{L^\infty} \leq C\varepsilon (1+t)^{-\mu}.
\end{eqnarray}
\end{Thm}
As we discuss in Remark \ref{decr}, the time decay rate of $L^\infty$-norm of the time and space derivatives of the global solution of \eqref{de0} with $\lambda=0$ is $-1$.
 Therefore, one can say that the time  decay rate  \eqref{dec}  is close to that with no damping case, when the decay in the damping coefficient becomes fast.
However, in the case that  the convergence rate of the limit $\lim_{x \rightarrow \pm \infty} (u_0 (x), v_0 (x))$ is sufficiently fast, the diffusion phenomena implies that
the slower the coefficient in the damping term decays,  the faster the solution does. One may expect that if initial data do not converge sufficiently fast to a constant state at spatial infinity,
the solution behaves like that of the equation of hyperbolic-type, not parabolic. 

\begin{Remark}\label{mainb}
Let  $\gamma >1$ and $( \phi,  \psi) \in C^1 _b (\R)$.  Suppose that $\mu >1$ and $\lambda >0$ or $\mu=1$ and $0 \leq \lambda \leq 2$.
In the same way as in the proof of Theorem \ref{maing}, we can shown that there exists a number $\varepsilon _0 >0$ such that if $0< \varepsilon \leq \varepsilon_0 $, then 
\begin{eqnarray} \label{ul}
&  T^* \geq  C\varepsilon^{-1}  \ \ \mbox{for} \  \mu>1 \ \mbox{and} \ \lambda\geq 0,\\
&  T^* C\varepsilon^{-\frac{2}{2-\lambda}}  \ \ \mbox{for} \ \mu=1  \ \mbox{and} \  0\leq \lambda < 2,\\
&  T^* \geq e^{\frac{C}{\varepsilon }}\ \ \mbox{for} \ \mu=1 \ \mbox{and} \ \lambda =2 .  
\end{eqnarray}
The sharp upper and lower estimates of the lifespan are given in the author's paper \cite{s5}.
\end{Remark}

\subsection*{Notations}

For $\Omega \subset \R^n$, $C^1 _b (\Omega)$ is the set of  bounded and continuous functions whose first partial derivatives are also bounded on $\Omega$.
The norm of $C^1 _b (\Omega)$  is $\|f \|_{C^1 _b (\Omega)} = \| f \|_{L^\infty  (\Omega)} + \| (\partial_{x_1} f,\ldots,\partial_{x_n}f) \|_{L^\infty (\Omega)}$.
When $\Omega = \R$, for abbreviation, we denote $\|\cdot  \|_{L^\infty (\Omega)} $    by $ \|\cdot  \|_{L^\infty}$.

\section{Preliminary}

First, we introduce some useful identities for the Riemann invariant, which are based on Lax's formulas in \cite{lax}.
For $c=\sqrt{-p' (u)}$, the plus and minus characteristic curves are solutions to the following deferential equations:
\begin{eqnarray*}
\frac{d x_{\pm}}{dt}(t) = \pm  c(u(t,x_{\pm} (t))).
\end{eqnarray*} 
For the solution $(u,v)$ constructed in Proposition \ref{local}, $c(u(t,x))$ is continuous on $[0,T^*) \times \R$.
Furthermore, for arbitrarily fixed   $T' \in (0, T^*)$, $c(u(t,x))$ is Lipschitz continuous with $x$ for all $t \in [0,T']$,
since $u_x$ and $c'(u)$ are uniformly bounded on  $ \R \times [0,T']$ from the definition of the lifespan $T^*$.
Therefore, if we take initial data $x_{\pm} (t^*) =x^*  \in \R $ with $t^* \in [0,T^*)$, the solution $x_{\pm} (t)$ 
uniquely exists on $[0,T^* )$ from a standard theorem for the existence and uniqueness for ordinary differential equations.

We can easily check that Riemann invariants $r$ and $s$ (see \eqref{ri} for their definitions) are solutions to the  following 1st order hyperbolic system
\begin{eqnarray}\label{ww}\left\{
\begin{array}{ll} 
\partial_- r =-\dfrac{\lambda}{2(1+t)^\mu}(r+s),\\
 \partial_+ s =-\dfrac{\lambda}{2(1+t)^\mu}(r+s),
\end{array}\right.
\end{eqnarray}
where $\partial_{\pm} = \partial_t \pm c \partial_x$. We put $A(t)=\exp(\int_0 ^t \frac{\lambda}{2(1+\tau)^\mu} d\tau)$.
 These equations can be written as
\begin{eqnarray}\label{ww2}\left\{
\begin{array}{ll} 
\partial_- ( A(t) r) = - \dfrac{\lambda A(t)}{2(1+t)^\mu}s ,\\
\partial_+ (A(t) s) =- \dfrac{\lambda A(t)}{2(1+t)^\mu}r .
\end{array}\right.
\end{eqnarray}
While, differentiating the equations in \eqref{ww} with $x$, from the identity $s_x - r_x = 2 \eta_x =-2c u_x$, we have
\begin{eqnarray}\label{rsx}\left\{
\begin{array}{ll} 
\partial_- r_x =\dfrac{c'}{2c}r_x (s_x -r_x) - \dfrac{\lambda }{2(1+t)^\mu} ( r_x + s_x)\\
 \partial_+  s_x = \dfrac{c'}{2c}r_x (r_x -s_x)- \dfrac{\lambda }{2(1+t)^\mu} ( r_x + s_x).
\end{array}\right.
\end{eqnarray}
Multiplying  the both side of the equations in \eqref{rsx} by $A(t) \sqrt{c}$, since
$$\partial_{-} \sqrt{c}=\frac{c'}{2\sqrt{c}} s_x \ \  \mbox{and} \ \  \partial_{+} \sqrt{c}=\frac{c'}{2\sqrt{c}} r_x ,$$
we have
\begin{eqnarray} \label{rs}\left\{
\begin{array}{ll} 
\partial_- y =-A(t)^{-1} \dfrac{\gamma +1}{4}u^{\frac{\gamma -3}{4}}y^2  - \dfrac{\lambda q}{2(1+t)^{\mu}},\\
\displaystyle \partial_+ q =-A(t)^{-1}  \dfrac{\gamma +1}{4}u^{\frac{\gamma -3}{4}}q^2  - \dfrac{\lambda y }{2(1+t)^{\mu}},
\end{array}\right.
\end{eqnarray}
where $y= A(t) \sqrt{c} r_x$ and $q= A(t) \sqrt{c} s_x$. 
Now we rewrite \eqref{rs} as  integral equalities.
We define $\theta_\gamma (u)$ as follows:
\begin{eqnarray*} 
\theta_{\gamma} (u) =\left\{ \begin{array}{ll}
\frac{4}{3-\gamma}  u^{\frac{3-\gamma}{4}}- \frac{4}{3-\gamma} \ \ \mbox{for} \ \gamma \not= 3,\\
\log u \ \ \mbox{for} \  \gamma = 3.
\end{array}
\right. 
\end{eqnarray*}
Since it holds that 
\begin{eqnarray}
\sqrt{c}s_x (t,x_{-} (t)) =   \frac{d}{dt}\theta_{\gamma}  (u(t,x_{-} (t))), \label{s-1}
\end{eqnarray}
from integration by parts, we  have
\begin{align}
&\int_0 ^t \dfrac{q(t,x_{-} (\tau))}{(1+\tau)^{\mu}} d\tau \notag \\
=&\int_0 ^t \dfrac{ A(\tau)}{(1+\tau)^{\mu}} \left( \dfrac{\mu}{(1+\tau)} - \dfrac{\lambda }{2(1+\tau)^{\mu}} \right) 
\theta_{\gamma} (\tau,x_{-} (\tau))d\tau \notag \\
& +\left( \dfrac{A(t) \theta_{\gamma}(t,x_{-}(t)) }{(1+t)^{\mu}}- \theta_{\gamma} (0,x_{-} (0))\right). \label{y-int}
\end{align}
From the first equation in \eqref{rs} and \eqref{y-int}, $y$  can be written on the minus characteristic curve through $(t,x)$ as follows:
\begin{align}
y(t,x)  =& y(0,x_{-} (0))\\
& -\int_0 ^t \dfrac{\lambda A(\tau)}{2(1+\tau)^{\mu}} \left( \dfrac{\mu}{(1+\tau)} - \dfrac{\lambda }{2(1+\tau)^{\mu}} \right) \theta_{\gamma} (\tau,x_{-} (\tau))d\tau \notag \\
& - \frac{\lambda}{2} \left( \dfrac{A(t) \theta_{\gamma}(t,x_{-}(t)) }{(1+\tau)^{\mu}}- \theta_{\gamma} (0,x_{-} (0))\right) \notag \\
&-\int_0 ^t A(\tau)^{-1} \dfrac{\gamma +1}{4}u^{\frac{\gamma -3}{4}}y^2 (\tau,x_{-} (\tau)) d\tau. \label{eq1}
\end{align}
In the same way as above, we can obtain the similar identity for $q$.  
Next, we introduce  key inequalities  which control  $u$  and $v$.
\begin{Lemma} \label{esP} 
Let $\gamma >1$, $\mu \geq 0$   and $\lambda \geq 0$. The following estimate holds for $C^1$ solution of \eqref{de0} constructed by Proposition \ref{local}
\begin{eqnarray*}
\|r (t) \|_{L^{\infty}} + \|s (t) \|_{L^{\infty}} \leq  \|r (0) \|_{L^{\infty}} + \|s (0) \|_{L^{\infty}}
\end{eqnarray*}
for $t \in [0,T^*)$.
\end{Lemma}
\begin{proof}
The proof of this lemma is almost same as in Lemma 7 in \cite{s5}.
We consider the characteristic curves $x_{-}(\cdot )$ and $x_+ (\cdot )$ through $(t, x)$.
Namely, we solve the characteristic equations with initial data $x_{\pm}(t)=x.$
As mentioned above, the characteristic curves exist on $[0,t]$, since it holds that $u, v \in C^1 _b ([0,T^*) \times \R)$  and that  $p'(u)$ is bounded on $[0,t] \times \R$  from Proposition \ref{local}.
From \eqref{ww2},  $r$ and $s$ can be written by
\begin{align*}
A(t) r(t,x) =&r(0,x_{-} (0))-\int_0 ^t  \dfrac{\lambda A(\tau) s(\tau , x_{-}(\tau)) }{2(1+\tau)^{\mu}} d \tau  \\
A(t) s(t,x) =&s(0,x_{+} (0))-\int_0 ^t  \dfrac{\lambda A(\tau) r(\tau , x_{+}(\tau)) }{2(1+\tau)^{\mu}}  d \tau. 
\end{align*}
We put $\Phi (t) = \|r(t)\|_{L^\infty}+ \|s(t)\|_{L^\infty}$. 
Summing up the above equations and taking $L^{\infty}$-norm, we have that 
\begin{align} \label{4es-2}
A(t) \Phi(t)  \leq & \Phi(0) + \int_0 ^t  \dfrac{\lambda A(\tau) \Phi(\tau)  }{2(1+\tau)^{\mu}} d \tau.
\end{align}
By the Gronwall inequality, we have
\begin{eqnarray*}
\Phi(t)  \leq  \Phi(0) .
\end{eqnarray*}
Then we have the desired estimate.
\end{proof}
The following fundamental inequality plays an important role in the proof of the decay estimate \eqref{dec}.
\begin{Lemma} \label{decA}
Suppose that $0 \leq \mu <1$ and $\lambda >0$ or $\mu=1$ and $\lambda >2$. Then it holds that
\begin{eqnarray*}
A^{-1}(t)\int_0 ^t \frac{A(s)}{(1+s)^{2\mu}} ds \leq C (1+t)^{-\mu}.
\end{eqnarray*}
\end{Lemma}
\begin{proof}
In the case that $\mu=1$, $A(t)=(1+t)^{\lambda/2}$, from which, straightforward computation yields the desired estimate.
In the case that $\mu <1$, since 
$$
\frac{2(1+t)^{\mu}}{\lambda} \dfrac{d}{dt}A(t) =A(t), 
$$
from the integration by parts, we have that
\begin{align*}
 A^{-1}(t) \int_0 ^t \frac{A(s)}{(1+s)^{2\mu}} ds = & \frac{2}{\lambda}\left( \frac{1}{(1+t)^{\mu}} -A^{-1} (t) \right) \\
 & + \frac{2\mu}{\lambda} A^{-1} (t) \int_0 ^t \frac{A(s)}{(1+s)^{\mu +1}} ds. 
\end{align*}
For the third term in the right hand side of the above equality, we see that
\begin{align*}
 A^{-1} (t) \int_0 ^t \frac{A(s)}{(1+s)^{\mu +1}} ds = & A^{-1} (t)\left( \int_0 ^{\frac{t}{2}} + \int_{\frac{t}{2}} ^t  \right)  \frac{A(s)}{(1+s)^{\mu +1}} ds \\
 \leq & C\left( A^{-1}(t)A\left(\frac{t}{2}\right) + \frac{1}{(1+t)^{\mu}} \right). \\
  \leq & C_1 \left( \exp( -C_2(1+t)^{1-\mu})+ \frac{1}{(1+t)^{\mu}} \right). 
\end{align*}
Hence we have the desired estimate.
\end{proof}
\section{Proof of Theorem \ref{maing}}

From Lemma \ref{esP}, we have that  
$$
\|v (t)\|_{L^\infty} + \|u^{-(\gamma-1)/2} (t)-1\|_{L^\infty} \leq C\varepsilon .
$$
Hence we see that
$$
 (1+C\varepsilon )^{-\frac{2}{\gamma-1}} \leq u(t,x) \leq (1-C\varepsilon )^{-\frac{2}{\gamma-1}},
$$
which implies $\|u(t)-1\|_{L^\infty} \leq C \varepsilon ,$ if $\varepsilon>0$ is sufficiently small. 
Applying the above estimates to \eqref{eq1}, we have that
\begin{align*}
|y(t,x)| \leq & C\left( \varepsilon +  \int_0 ^t \frac{\varepsilon A(s)}{(1+s)^{2\mu}} ds \right. \\
 & \left. + \frac{ \varepsilon A(t)}{(1+t)^{\mu}} +  \int_0 ^t A^{-1}(s) y^2 (t,x_{-} (s)) ds \right)
\end{align*}
Multiplying the both side of the above inequality by $(1+t)^\mu /A(t)$ and putting $Y(T)=\sup_{[0,T]} (1+t)^{\mu} \|r_x (t) \|_{L^\infty} $, from the definition of $y$ and Lemma \ref{decA}, we see that
\begin{eqnarray*}
Y(T)\leq C_1 \varepsilon + C_2  Y^2 (T). 
\end{eqnarray*}
Since $Y(0) \leq C_3 \varepsilon $, we have the uniform estimate $Y(T) \leq 2(C_1 + C_3)\varepsilon $ for $T \in [0,T^*)$, if $\varepsilon $ is sufficiently small.
Similarly we can obtain the same uniform estimate of $s_x$ as for $r_x$, which implies the global existence ($T^* =\infty$) and the decay estimate
$$
\|(r_x, s_x) (t)\|_{L^\infty} \leq C \varepsilon (1+t)^{-\mu}.
$$
The decay estimate for $(r_t, s_t)$ can be shown by \eqref{rsx} and Lemma \ref{esP}.

\begin{Remark} \label{decr}
Here we discuss the global existence and a decay estimate of solutions of \eqref{de0} with $\lambda=0$.
In order to consider them, we review some formula of $r_x$ and $s_x$ in the case that $\lambda=0$ (cf. \cite{lax}).
Now we note that $A(t) \equiv 1$ if $\lambda=0$. Solving the differential equation \eqref{rs}, we have that
\begin{eqnarray*}
y(t,x_{-}(t))=\dfrac{1}{\dfrac{1}{y(0,x_{-}(0))}+\displaystyle \int_0 ^t\dfrac{\gamma +1}{4}u^{\frac{\gamma -3}{4}} (\tau , x_{-}(\tau))  d\tau}.
\end{eqnarray*}
Therefore we have that if $r_x (0,x), s_x (0,x) \geq 0$ and $\varepsilon $ is suitably small, then the \eqref{de0} has a global $C^1 _b$ solution such that
$$
\|(r_t, s_t)(t)\|_{L^\infty} + \|(s_x, r_x) (t)\|_{L^\infty} \leq C(1+t)^{-1}.
$$
\end{Remark}

\section*{Acknowledgments}
The author would like to thank the referee for carefully reading of the manuscript.
The research is  supported by  Grant-in-Aid for Young Scientists Research (B), No. 16K17631.


\end{document}